\documentclass[12pt]{article}
\usepackage[latin1]{inputenc}
\usepackage{amsmath, amsthm, amssymb}
\usepackage{latexsym}
\usepackage{mathrsfs}
\usepackage{epsfig}
   \newtheorem{theorem}{Theorem}
   \newtheorem{proposition}[theorem]{Proposition}
   
   \newtheorem{lemma}[theorem]{Lemma}
 \theoremstyle{definition}
   \newtheorem{definition}{Definition}

\newbox\qedbox
\setbox\qedbox=\hbox{$\Box$}
\newenvironment{Proof}{\smallskip\noindent{\bf Proof.}\hskip \labelsep}%
                        {\hfill\penalty10000\copy\qedbox\par\medskip}
\newenvironment{remark}{\smallskip\noindent{\bf Remark.}\hskip \labelsep}%
                        {\hfill\penalty10000\copy\qedbox\par\medskip}
\newenvironment{example}{\smallskip\noindent{\bf Example.}}%

\begin{document}

\centerline{\Large \bf Weierstrass representation for semi-discrete minimal}
\centerline{\Large \bf surfaces, and comparison of various discretized catenoids}

\medskip
\medskip

\centerline{\large Wayne Rossman, Masashi Yasumoto}

\medskip

\medskip

\begin{quote} {\bf Abstract:} 
We give a Weierstrass type representation for semi-discrete
minimal surfaces in Euclidean $3$-space.  We then give
explicit parametrizations of various smooth, semi-discrete
and fully-discrete catenoids, determined from either
variational or integrable systems principles.  Finally, we state the shared
properties that those various catenoids have.
\end{quote}

\medskip
\medskip
\medskip

\section{Introduction}

The well known minimal surface of revolution in
$\mathbb{R}^3=\{ (x_1,x_2, x_3)^t \ | \ x_j \in \mathbb{R} \}$ called the catenoid, which we refer to as
the "smooth catenoid" here and which can be parametrized by
\begin{equation}\label{eqn:smoothcat}
x(u,v) = \begin{pmatrix} \cosh u \cos v\\ \cosh u \sin v\\ u\end{pmatrix}, \ 
v \in [0,2 \pi ), \ 
u \in \mathbb{R},
\end{equation}
has a number of discretized versions. A fully discretized
version can be found in \cite{PR1} by Polthier and the
first author, which is defined using a variational approach,
that is, those surfaces are triangulated meshes that are critical for
area with respect to smooth variations of the vertex set.  A
different approach for defining fully discrete catenoids, using
quadrilateral faces and based on
integrable systems methods, was found by Bobenko and Pinkall
\cite{BP1}, \cite{BP2}.  Both approaches apply to much wider
classes of surfaces.

One can also consider semi-discrete catenoids, that is,
catenoids that are discretized in only one of the two
parameter directions corresponding to $u$ and $v$ in
\eqref{eqn:smoothcat}.  There are now four choices for how
to proceed with this, by choosing either the $u$ direction or
$v$ direction to discretize, and by choosing to use either
variational principles or integrable systems principles to
determine the discretizations.  Again, these
approaches apply to much wider classes of surfaces.

Here we compare these various smooth,
semi-discrete and fully-discrete catenoids to see in what
ways they do or do not coincide.  For the smooth
and fully-discrete catenoids, the parametrizations have
already been determined, making comparisons between
them elementary.  However, for some of the
semi-discrete cases, we will need to first establish those
parametrizations here.  In particular, we will provide a
Weierstrass representation for determining
semi-discrete minimal surfaces as defined by Mueller and
Wallner \cite{MW2}, \cite{Wa1}.
Construction of the semi-discrete catenoids in particular, via an integrable
systems approach, can be done either with this Weierstrass
representation, or without it (instead using
the results by Mueller and Wallner).
However, the usefulness of the Weierstrass representation
comes when one wishes to consider the full class of
semi-discrete minimal surfaces based on an integrable
systems approach, as this representation gives a
classification of such surfaces in terms of semi-discrete
holomorphic functions. This Weierstrass representation 
can be regarded as a restatement of the definition of such surfaces (Definition \ref{def-minimal}), but in a more explicit form 
that tells us how the surface is constructed from the given dual surface inscribed in a sphere.  

Once we have established this representation for
semi-discrete minimal surfaces (Theorem \ref{semi-discrete weierstrass}), we compare
the various types of catenoids (Theorem \ref{comparison}).

To make semi-discrete catenoids based on variational principles,
Machigashira \cite{Ma1} chose to discretize them in the $u$ direction.
He then classified these surfaces and studied their stability
properties.  The surfaces obtained by Machigashira will be seen 
(Proposition \ref{M-PR-limit}) to be limiting cases of the discrete
catenoids found in \cite{PR1}.

From the point of view of architectural structures in the shape of a
semi-discrete catenoid, Machigashira's catenoids would involve
producing circular-shaped flat pieces that cannot be so efficiently
made as cut-outs from planar sheets, since there would be a large
amount of waste material.  So from the architectural point of view,
a more efficient use of materials would be to discretize in
the $v$ direction instead.  Such semi-discrete catenoids are considered
here as well.

To distinguish between various catenoids, we write the superscript $va$ (resp. $in$) when the catenoid is constructed by a variational (resp. integrable systems) 
approach, and write the subscript $pd$ (resp. $ps$) when the catenoid has a discrete profile curve (resp. smooth profile curve) and the subscript $rd$ (resp. $rs$) 
when the catenoid is discrete (resp. smooth) in the rotational direction. Thus, in total, we consider the seven types of catenoids in Table \ref{catenoids}.

For catenoids with discrete profile curves, we will assume them to have a "neck vertex". In other words, we assume there exists a plane 
of reflective symmetry of the catenoids that is perpendicular to the axis of rotation symmetry and also contains one vertex of each profile curve. We note that there 
do exist discrete catenoids that do not have this neck-vertex symmetry.

\begin{table}\label{catenoids}
\begin{center}
\begin{tabular}{|c|c|} \hline
 & associated authors \\ \hline
smooth catenoid & (classically known surface) \\ \hline
$BP_{pd,rd}^{in}$-catenoid & Bobenko and Pinkall \\ \hline
$PR_{pd,rd}^{va}$-catenoid & Polthier and Rossman \\ \hline
$M_{pd,rs}^{va}$-catenoid & Machigashira \\ \hline
$MW_{pd,rs}^{in}$-catenoid & Mueller and Wallner \\ \hline
$MW_{ps,rd}^{in}$-catenoid & Mueller and Wallner \\ \hline
$M_{ps,rd}^{va}$-catenoid & (Machigashira analogue) \\ \hline
\end{tabular}
\caption{Names of seven types of catenoids}
\end{center}
\end{table}

\begin{theorem}\label{comparison}
After appropriate normalizations, we have the following:
\begin{enumerate}
\item $PR^{va}_{pd,rd}$-catenoid profile curves and $M^{va}_{pd,rs}$-catenoid profile curves are never the same, but 
$PR^{va}_{pd,rd}$-catenoid profile curves converge to $M^{va}_{pd,rs}$-catenoid profile curves as the angle 
of rotation symmetry approaches $0$. 
\item $BP^{in}_{pd,rd}$-catenoids and $MW^{in}_{pd,rs}$-catenoids have the same profile curves.
\item $BP^{in}_{pd,rd}$-catenoid ($MW^{in}_{pd,rs}$-catenoid) profile curves and $PR^{va}_{pd,rd}$-catenoid profile curves are never the same, 
and $BP^{in}_{pd,rd}$-catenoid ($MW^{in}_{pd,rs}$-catenoid) profile curves 
and $M^{va}_{pd,rs}$-catenoid profile curves are never the same.
\item The smooth catenoid and $MW^{in}_{ps,rd}$-catenoid have the same profile curve.
\item $M^{va}_{ps,rd}$-catenoid profile curves and the smooth catenoid's profile curve are never the same. 
$M^{va}_{ps,rd}$-catenoid profile curves converge to the smooth catenoid ($MW^{in}_{ps,rd}$-catenoid) profile curve as the angle 
of rotation symmetry approaches $0$.
\item For all types of catenoids, the profile curves have vertices lying on affinely scaled graphs of the hyperbolic cosine function.
\end{enumerate}
\end{theorem}

\section{Notation for semi-discrete surfaces}
To consider semi-discrete minimal surfaces from an integrable systems approach, we set some notations in this section.

Let $x=x(k,t)$ be a map from a domain in $\mathbb{Z} \times \mathbb{R}$ to $\mathbb{R}^3$ $(k \in \mathbb{Z},$ $t \in \mathbb{R})$. 
We call $x$ a {\em semi-discrete surface}. Set
\[
\partial x=\frac{\partial x}{\partial t}, \ \Delta x =x_1 -x, \ \partial \Delta x=\partial x_1 -\partial x,
\]
where $x_1=x(k+1,t)$. The following definitions can be found in \cite{MW2}, and are all naturally motivated by geometric properties found in 
previous works, such as \cite{BP1}, \cite{BP2}, \cite{BPW1}, \cite{H-J1}, \cite{H-JHP1}, \cite{Ho1}, \cite{MW1}, \cite{MW2}, \cite{Wa1}.

\begin{definition} Let $x$ be a semi-discrete surface.
\begin{itemize}
\item $x$ is {\em conjugate} if $\partial x$, $\Delta x$ and $\partial \Delta x$ are linearly dependent.
\item $x$ is {\em circular} if there exists a circle $\mathscr{C}$ passing 
through $x$ and $x_1$ that is tangent to $\partial x$, $\partial x_1$ there (for all $k,$ $t$).
\end{itemize}
\end{definition}

\begin{remark}\label{circular-plane}
If $x$ lies in $\mathbb{R}^2 \cong \mathbb{C}$, circularity is equivalent to the following condition: there exists a function $s \in \mathbb{R} \setminus \{ 0\}$ such that
\begin{equation}\label{circular-condition}
\Delta x ={\rm i}s\left( \frac{\partial x}{\| \partial x \|}+\frac{\partial x_1}{\| \partial x_1 \|}\right) .
\end{equation}
\end{remark}

\begin{definition}\label{duality}
Suppose $x$, $x^*$ are conjugate semi-discrete surfaces. Then $x$ and $x^*$ are {\em dual} surfaces if there exists a function $\nu : \mathbb{Z} \times \mathbb{R} \rightarrow \mathbb{R}^{+}$ so that 
\[
\partial x^* =-\frac{1}{\nu ^2}\partial x, \ \Delta x^* =\frac{1}{\nu \nu_1}\Delta x.
\]
\end{definition}

\begin{definition}\label{semi-iso}
A circular semi-discrete surface $x$ is {\em isothermic} if there exist positive functions $\nu, \ \sigma, \ \tau$ such that
\[
\| \Delta x\| ^2 =\sigma \nu \nu_1,\ \| \partial x\| ^2 = \tau \nu^2 , \ {\rm with} \ \partial \sigma = \Delta \tau ={\rm 0}.
\]
\end{definition}

\begin{remark}
The $\nu$, $\nu_1$ in Definitions \ref{duality} and \ref{semi-iso} are the same, by the proof of Theorem 11 in \cite{MW2}.  
\end{remark}

\begin{definition}\label{def-minimal}
A semi-discrete isothermic surface $x$ is {\em minimal} if $x^*$ is inscribed in a sphere.
\end{definition}

\section{Semi-discrete catenoids with discrete profile curve} 

Take the following parametrization for $MW_{pd,rs}^{in}$-catenoids:
\[
x(k,t)=\begin{pmatrix}f(k){\rm cos}t \\ f(k){\rm sin}t \\ h\cdot k \end{pmatrix}
\]
where $f=f(k)$ and $h$ are positive. Then, with $f_1=f(k+1)$,
\[
\| \Delta x\| ^2 =(f_1-f)^2 +h^2,
\]
\[
\| \partial x\| ^2 = f^2.
\]
One can check that $x$ is isothermic by taking 
\[
\nu =f, \ \ \tau =1 \ \ {\rm and \ \ } \sigma=\frac{(\Delta f)^2 +h^2}{f \cdot f_1}. 
\]
We compute $x^*$ by solving
\[
x^* =-\frac{1}{\nu^2}\int \partial x dt
\]
\[
=\begin{pmatrix} -\frac{1}{f}{\rm cos}t\\ -\frac{1}{f}{\rm sin}t\\ 0 \end{pmatrix}+ \overrightarrow{c}_k,
\]
where $\overrightarrow{c}_k$ depends on $k$ but not $t$. We now have
\[
\Delta x^* =\frac{1}{f\cdot f_1}\begin{pmatrix} \Delta f \cdot {\rm cos}t\\ \Delta f \cdot {\rm sin}t\\ h \end{pmatrix}
\]
\[
=\frac{1}{f\cdot f_1}\begin{pmatrix} \Delta f \cdot {\rm cos}t\\ \Delta f \cdot {\rm sin}t\\ 0 \end{pmatrix}+(\overrightarrow{c}_{k+1}-\overrightarrow{c}_k)
\]
\[
\Leftrightarrow \overrightarrow{c}_{k+1}-\overrightarrow{c}_k =\begin{pmatrix}0 \\ 0 \\ \frac{h}{f\cdot f_1} \end{pmatrix}.
\]
So without loss of generality, we can take $\overrightarrow{c}_k$ as
\[
\overrightarrow{c}_k=\begin{pmatrix}0 \\ 0 \\ c(k) \end{pmatrix}.
\]
For $x$ to be minimal, we wish to have
\[
\| x^* \| \equiv {\rm constant}
\]
for some choice of $c(0)$. We obtain the following system of difference equations:
\begin{eqnarray}
f(k+1)=h c(k)f(k)+\sqrt{(h c(k)f(k))^2 +f(k)^2+h^2}, \label{MW-recursion1} \\
c(k+1)=c(k)+\frac{h}{f(k)f(k+1)} \label{MW-recursion2}
\end{eqnarray}
with initial conditions $f(0)$ and $c(0)$. (Equation (\ref{MW-recursion1}) follows from substituting Equation (\ref{MW-recursion2}) into the equation given by 
$\| x_1 \| ^2 =\| x \| ^2$.) Then we can solve recursively for $f(k)$ and $c(k)$. See Figure \ref{MW-catenoid}.

\begin{figure}
 \begin{center}
  
 \includegraphics[width=60mm]{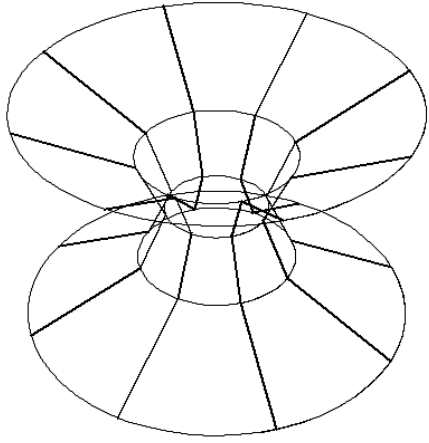}
  
 \end{center}
 \caption{a semi-discrete $MW^{in}_{pd,rs}$-catenoid with discretized profile curve}
 \label{MW-catenoid}
\end{figure}

In order to compare the other catenoids with $MW^{in}_{pd,rs}$-catenoids, we wish to reduce the above system of difference equations to one equation.
\begin{lemma}\label{MW-const}
We have
\begin{equation}\label{MW-recursion}
f(k+2)=\frac{f(k+1)^2 + h^2}{f(k)},
\end{equation}
with initial conditions $f(0)$ and $f(1)$ determined by $f(0)$ and $c(0)$.
\end{lemma}

\begin{proof}
By Equation (\ref{MW-recursion1}),
\[
f(k+2)^2 -2h c(k+1) f(k+1)f(k+2)-f(k+1)^2 +h^2=0.
\]
Inserting Equation (\ref{MW-recursion2}) into the above equation, we have
\begin{equation} \label{rec-eq1}
\begin{split}
f(k)f(k+2)^2 -2h(c(k)f(k)f(k+1)+h)f(k+2) \\
-f(k)(f(k+1)^2+h^2)=0. 
\end{split}
\end{equation}
Again by Equation (\ref{MW-recursion1}),
\begin{equation}
c(k)f(k)f(k+1)=\frac{f(k+1)^2 -f(k)^2-h^2}{2h}. \label{rec-eq2}
\end{equation}
Inserting (\ref{rec-eq2}) into (\ref{rec-eq1}), and noting that $f(k)$ for all $k$,
\[
f(k+2)=\frac{f(k+1)^2 + h^2}{f(k)}.
\]
\end{proof}

Lemma \ref{MW-const} implies
\[
f(1)f(-1)=f(0)^2 +h^2,
\]
and then neck-vertex symmetry gives $f(1)=f(-1)$, and so
\[
f(1)^2=f(0)^2 +h^2.
\]
Then Equation (\ref{MW-recursion1}) implies
\[
\sqrt{h^2 +f(0)^2}=h c(0)f(0)+\sqrt{(h c(0)f(0))^2 +f(0)^2+h^2},
\]
implying $hc(0)f(0)=0$, and so $c(0)=0$. Without loss of generality, we can take $f(0)=1$, and then the solution to Equation (\ref{MW-recursion}) 
is
\[
f(k)=\cosh ({\rm arcsinh}(h) \cdot k).
\]

\section{Semi-discrete catenoids foliated by discrete circles}
Take the following parametrization for $MW_{ps,rd}^{in}$-catenoids:
\[
x(k,t)=\begin{pmatrix}f(t){\rm cos}\alpha k\\ f(t){\rm sin}\alpha k\\ t \end{pmatrix},
\]
where $f(t)$ and $\alpha$ are positive. We assume
\begin{equation}\label{initial-condition}
f(0)=1 \ {\rm and \ } f^{\prime}(0)=0.
\end{equation}
Then
\[
\| \Delta x\| ^2 =4f(t)^2{\rm sin}^2\frac{\alpha}{2},
\]
\[
\| \partial x\| ^2 = (f^{\prime}(t))^2 +1.
\]
One can confirm that $x$ is isothermic by taking 
\[
\nu=f(t), \ \ \tau =\frac{(f^{\prime}(t))^2 +1}{(f(t))^2} \ \ {\rm and \ \ } \sigma =4\sin ^2\frac{\alpha}{2}. 
\]
Now,
\[
x^* =-\int \frac{1}{\nu^2}\partial x dt
\]
\[
=\begin{pmatrix}-{\rm cos}\alpha k \int \frac{f^{\prime}}{f^2}dt\\ -{\rm sin}\alpha k \int \frac{f^{\prime}}{f^2}dt\\ - \int \frac{1}{f^2}dt \end{pmatrix}
\]
\begin{equation}\label{result1}
=\begin{pmatrix}\frac{{\rm cos}\alpha k}{f}\\ \frac{{\rm sin}\alpha k}{f}\\ \ell (t) \end{pmatrix}+ \overrightarrow{c}_k,
\end{equation}
where $\overrightarrow{c}_k$ depends on $k$ but not $t$, and $\ell (t)=\int_{0}^{t} \frac{1}{(f(\tilde{t} ))^2}d\tilde{t}$ depends on $t$ but not $k$. We compute that (line (\ref{result2}) follows from the definition of $x^*$ and line (\ref{result3}) follows from (\ref{result1}))
\begin{equation}\label{result2}
\Delta x^* =\frac{1}{f}\begin{pmatrix}\Delta {\rm cos}\alpha k\\ \Delta {\rm sin}\alpha k\\ 0 \end{pmatrix}
\end{equation}
\begin{equation}\label{result3}
=\frac{1}{f}\begin{pmatrix} \Delta {\rm cos}\alpha k\\ \Delta {\rm sin}\alpha k \\ 0 \end{pmatrix}+(\overrightarrow{c}_{k+1}-\overrightarrow{c}_k)
\end{equation}
\[
\Leftrightarrow \overrightarrow{c}_{k+1}-\overrightarrow{c}_k=\begin{pmatrix}0 \\ 0 \\ 0 \end{pmatrix}.
\]
Thus $\overrightarrow{c}_k =\overrightarrow{c}_0$ for all $k$. For $x$ to be minimal, $x^*$ must be inscribed in a sphere and therefore we can choose $\overrightarrow{c}_0 =\begin{pmatrix} 0 \\ 0 \\ c_0\end{pmatrix}$ so that $\| x^* \|$ is constant. From (\ref{result1}) we have

\[
\frac{1}{f^2}+\left( \int_{0}^{t} \frac{1}{(f(\tilde{t} ))^2}d\tilde{t} +c_0 \right) ^2={\rm constant}
\]
\[
\Rightarrow \int_{0}^{t} \frac{1}{(f(\tilde{t} ))^2}d\tilde{t}=\frac{f^{\prime}}{f} -c_0
\]
\[
\Rightarrow f^{\prime \prime}f-f^{\prime}=1
\]
with initial value $f(0)$, and we find from (\ref{initial-condition}) that
\begin{center}
$f(t)=\cosh t$ (and $c_0=0$).
\end{center}

Thus semi-discrete catenoids with smooth profile curves and fixed $\alpha$ are unique up to homotheties. The 
picture in Figure \ref{MW'catenoid} is such a semi-discrete catenoid. In fact, we have proven:

\begin{proposition}
The profile curve of $MW^{in}_{ps,rd}$-catenoids is independent of choice of $\alpha$.
\end{proposition}

\begin{figure}
 \begin{center}

 \includegraphics[width=60mm]{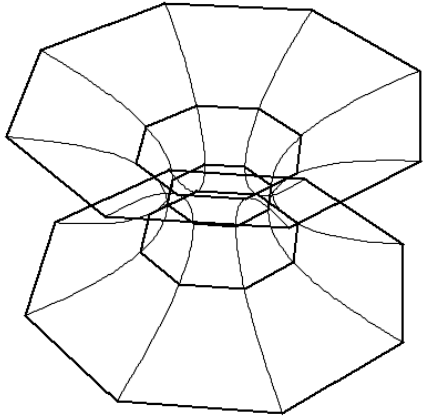}
  
 \end{center}
 \caption{a semi-discrete $MW^{in}_{ps,rd}$-catenoid discretized in the direction of rotation}
 \label{MW'catenoid}
\end{figure}

\section{Weierstrass representation for semi-discrete minimal surfaces}

We now give a Weierstrass representation for semi-discrete minimal surfaces. First we define semi-discrete holomorphic functions.
\begin{definition}
A semi-discrete isothermic surface $g$ is a {\em semi-discrete holomorphic function} if $g:\mathbb{Z} \times \mathbb{R}$ lies in a plane.
\end{definition}

\begin{remark}
Semi-discrete holomorphic maps have the following property: With $\sigma$ and $\tau$ as in Definition \ref{semi-iso} (with $x$ replaced by $g$),
\begin{equation}\label{property}
\frac{\| \Delta g \| ^2}{\| g^{\prime} \| \ \| g_1^{\prime} \|}=\frac{\sigma}{\tau},
\end{equation}
where $g^{\prime}=\partial g$. So we can think of $\tau$ and $\sigma$ as the (absolute values of the) semi-discrete cross-ratio factorizing functions, in analogy to the fully 
discrete case in \cite{BP1}, \cite{BP2}, \cite{H-J1}, \cite{H-JHP1}, \cite{Ho1}, \cite{Ro1}.
\end{remark}

We introduce the following recipe for constructing semi-discrete minimal surfaces.
\begin{theorem}[Weierstrass representation] \label{semi-discrete weierstrass}
We can construct a semi-discrete minimal surface from a semi-discrete holomorphic function $g$ by solving
\begin{equation}\label{weierstrass-formula}
\partial x =-\frac{\tau}{2} {\rm Re} \begin{pmatrix} \frac{1-g^2}{g^{\prime}}\\ \frac{{\rm i}(1+g^2)}{g^{\prime}} \\ \frac{2g}{g^{\prime}} \end{pmatrix}, \ \
\Delta x =\frac{\sigma}{2} {\rm Re} \begin{pmatrix} \frac{1-g g_1}{\Delta g}\\ \frac{{\rm i}(1+g g_1)}{\Delta g}\\ \frac{g+g_1}{\Delta g} \end{pmatrix}
\end{equation}
with $\tau$ and $\sigma$ as in Definition \ref{semi-iso}.
Conversely, any semi-discrete minimal surface is described in this way by some semi-discrete holomorphic function $g$.
\end{theorem}

\begin{Proof}
We start by proving the first sentence of Theorem \ref{semi-discrete weierstrass}.
\[
x^* :=\begin{pmatrix}\frac{2g}{1+\| g\| ^2}\\ \frac{-1+\| g\| ^2}{1+\| g\| ^2}\end{pmatrix} \in \mathbb{S}^2 \subset \mathbb{C} \times \mathbb{R}=\mathbb{R}^3
\]
is semi-discrete isothermic, because $x^*$ is the image of $g$ under the inverse of 
stereographic projection. Then
\[
\partial x^* =\frac{2}{(1+\| g \| ^2)^2}\begin{pmatrix} \ g^{\prime}-\overline{g^{\prime}}g^2 \\ g^{\prime}\bar{g}+\overline{g^{\prime}}g \end{pmatrix},
\]
\[
\Delta x^* =\frac{2}{(1+\| g \| ^2)(1+\| g_1 \| ^2)}\begin{pmatrix} \Delta g-\overline{\Delta g}g g_1\\ \Delta g\bar{g_1}+\overline{\Delta g}g\end{pmatrix}.
\]
It follows that
\[
\| \partial x^* \| ^2 =\frac{4\| g^{\prime} \|^2}{(1+\| g \| ^2)^2}=\frac{4 \tau \nu^2}{(1+\| g \| ^2)^2},
\]
\[
\| \Delta x^* \| ^2 =\frac{4\| \Delta g \|^2}{(1+\| g \| ^2)(1+\| g_1 \| ^2)}=\frac{4\sigma \nu \nu_1}{(1+\| g \| ^2)(1+\| g_1 \| ^2)},
\]
so we can take the data $\tau^*$, $\sigma^*$, $\nu^*$ for the isothermic surface $x^*$ to be 
\[
\tau^*=\tau , \  \sigma^*=\sigma , \ \nu^* =\frac{2 \nu}{1+\| g \|^2}.
\]
$\sigma^*$ depends only on $k$, and $\tau^*$ depends only on $t$. Then

\[
\frac{-1}{(\nu^*) ^2}\partial x^*= \frac{-1}{2\nu^2}\begin{pmatrix} g^{\prime}-\bar{g^{\prime}}g^2\\ g^{\prime}\bar{g}+\bar{g^{\prime}}g \end{pmatrix}
= \frac{-\tau}{2\| g^{\prime}\|^2}\begin{pmatrix} g^{\prime}-\bar{g^{\prime}}g^2\\ g^{\prime}\bar{g}+\bar{g^{\prime}}g\end{pmatrix}
\]
\[
= -\frac{\tau}{2} \begin{pmatrix} {\rm Re}\left( \frac{1-g^2}{g^{\prime}} \right) +{\rm i Re}\left( \frac{{\rm i}(1+g^2)}{g^{\prime}} \right)\\ {\rm Re}\left( \frac{2g}{g^{\prime}} \right) \end{pmatrix}=\partial x,
\]
where we have identified $\mathbb{C} \times \mathbb{R}$ and $\mathbb{R}^3$ in the final equality, and similarly,
\[
\frac{1}{\nu^* \nu_1^*}\Delta x^{*}= \frac{1}{2 \nu \nu_1}\begin{pmatrix} \Delta g-\overline{\Delta g}g g_1\\ \Delta g\overline{g_1}+\overline{\Delta g}g \end{pmatrix}
\]
\[
= \frac{\sigma}{2\| \Delta g \|^2}\begin{pmatrix} \Delta g-\overline{\Delta g}g g_1\\ \Delta g\overline{g_1}+\overline{\Delta g}g \end{pmatrix}
\]
\[
= \frac{\sigma}{2} \begin{pmatrix} {\rm Re} \left( \frac{1-g g_1}{\Delta g} \right) + {\rm iRe}\left( \frac{{\rm i}(1+g g_1)}{\Delta g} \right) \\ {\rm Re}\left( \frac{g+g_1}{\Delta g} \right) \end{pmatrix} =\Delta x.
\]
Thus if $x$ exists solving (\ref{weierstrass-formula}), $x$ and $x^*$ are dual to each other. A direct computation shows
\[
\| \Delta x \| ^2 =\sigma \left( \frac{1+\| g\| ^2}{2 \nu } \right) \left( \frac{1+\| g_1\| ^2}{2 \nu _1} \right) ,
\]
\[
\| \partial x\| ^2 =\tau \left( \frac{1+\| g\| ^2}{2 \nu } \right)^2 ,
\]
so $x$ would be isothermic if it were circular. Since $x^*$ is inscribed in a sphere, $x$ would then be a semi-discrete minimal surface. Thus it remains to check existence and circularity of $x$.

To show existence of $x$, we need to show compatibility of the two equations in (\ref{weierstrass-formula}), and this amounts to showing that the two operators $\Delta$ and $\partial$ in (\ref{weierstrass-formula}) commute, that is,

\begin{equation}\label{compatibility}
\partial \left( \frac{\sigma}{2} {\rm Re} \begin{pmatrix} \frac{1-gg_1}{\Delta g} \\ \frac{{\rm i}(1+gg_1)}{\Delta g} \\ \frac{g+g_1}{\Delta g} \end{pmatrix} \right) =\Delta \left( -\frac{\tau}{2}{\rm Re}\begin{pmatrix} \frac{1-g^2}{g^{\prime}} \\ \frac{{\rm i}(1+g^2)}{g^{\prime}} \\ \frac{2g}{g^{\prime}} \end{pmatrix}
\right) .
\end{equation}
One can compute (where circularity of $g$, i.e. (\ref{circular-condition}) applied to $g$ instead of $x$, is used toward the end, and (\ref{property}) is used 
toward the beginning)
\begin{eqnarray}
&& \text{Left-hand side of (\ref{compatibility})} \nonumber \\
&=&\frac{\sigma}{2}{\rm Re} \left( \frac{1}{(\Delta g)^2} \begin{pmatrix} g^2 g^{\prime}-g_1^{\prime}-g^{\prime}g_1^2+g^{\prime} \\ {\rm i}(g^{\prime}g_1^2+g^{\prime}-g^2g_1^{\prime}-g_1^{\prime}) \\ 2g^{\prime}g_1-2gg_1^{\prime} \end{pmatrix} \right) \nonumber \\
&=&\frac{\tau \| \Delta g \|^2}{2\| g^{\prime}\| \| g_1^{\prime}\|}{\rm Re} \left( \frac{1}{(\Delta g)^2} \begin{pmatrix} g^2 g^{\prime}-g_1^{\prime}-g^{\prime}g_1^2+g^{\prime} \\ {\rm i}(g^{\prime}g_1^2+g^{\prime}-g^2g_1^{\prime}-g_1^{\prime}) \\ 2g^{\prime}g_1-2gg_1^{\prime} \end{pmatrix} \right) \nonumber \\
&=&\frac{\tau}{2} {\rm Re} \left( \frac{\overline{\Delta g}}{\| g^{\prime}\| \| g_1^{\prime} \| \Delta g}\begin{pmatrix} g^2 g^{\prime}-g_1^{\prime}-g^{\prime}g_1^2+g^{\prime} \\ {\rm i}(g^{\prime}g_1^2+g^{\prime}-g^2g_1^{\prime}-g_1^{\prime}) \\ 2g^{\prime}g_1-2gg_1^{\prime} \end{pmatrix} \right) \nonumber \\
&=&-\frac{\tau}{2}{\rm Re}\left( \frac{1}{g^{\prime} g_1^{\prime}}\begin{pmatrix}  g^2 g^{\prime}-g_1^{\prime}-g^{\prime}g_1^2+g^{\prime} \\ {\rm i}(g^{\prime}g_1^2+g^{\prime}-g^2g_1^{\prime}-g_1^{\prime}) \\ 2g^{\prime}g_1-2gg_1^{\prime} \end{pmatrix} \right). \nonumber \\
&=& \text{Right-hand side of (\ref{compatibility}).} \nonumber
\end{eqnarray}

The last task is to check that $x$ is circular. In order to prove this, we use the next lemma. We define $\hat{g}$ as
\[
\hat{g}:=\frac{pg+q}{-\bar{q}g+\bar{p}},\ 
{\rm for \ some\ \ } A(p,q)=\begin{pmatrix} p & q \\ -\bar{q} & \bar{p} \end{pmatrix} \in {\rm SU_2}/\{ \pm {\rm I} \}.
\]
$\hat{g}$ satisfies $|\hat{g}^{\prime}|^2=\tau \hat{\nu}^2 ,$ $|\Delta \hat{g}|=\sigma \hat{\nu} \hat{\nu}_1$, where $\hat{\nu}=\frac{\nu}{| -\bar{q}g+p |^2}$.

\begin{lemma}\label{rotation}
Let $\hat{x}$ be a semi-discrete isothermic surface satisfying (\ref{weierstrass-formula}), but with x replaced by $\hat{x}$ and g replaced by $\hat{g}$. Then there 
exists a matrix $A=A(p,q) \in {\rm SO_3}(\mathbb{R})$ so that
\begin{center}
$\partial \hat{x}=A \partial x$, $\Delta \hat{x}=A \Delta x$.
\end{center}
Therefore, x and $\hat{x}$ differ by only a rotation of $\mathbb{R}^3$. Furthermore, any $A \in {\rm SO_3}(\mathbb{R})$ can be obtaining with appropriate selection 
of p and q.
\end{lemma}

We prove Lemma \ref{rotation} after this proof of Theorem \ref{semi-discrete weierstrass}. 

So without loss of generality, by Lemma \ref{rotation}, we can rotate and translate $x$ so that span$\{ \partial x,\partial x_1,\Delta x \} = \mathbb{C} \times \{ 0 \}$ 
for one edge $\overline{xx_1}$. Then
\begin{center}
$\frac{g}{g^{\prime}}$, $\frac{g_1}{g_1^{\prime}}$, $\frac{g+g_1}{\Delta g} \in $i$\mathbb{R}$.
\end{center}
Setting $g=re^{{\rm i}\theta}$ for $r=r(k,t) \geq 0$, $\theta =\theta (k,t) \in \mathbb{R}$, we have
\[
r^{\prime}=r_1^{\prime}=0, \ r_1e^{{\rm i} \theta_1}+re^{{\rm i} \theta}={\rm i} \rho (r_1e^{{\rm i} \theta_1}-re^{{\rm i} \theta})
\]
for some $\rho \in \mathbb{R}$. Taking the absolute value of
\[
r_1({\rm i}\rho -1)e^{{\rm i}\theta_1}=r({\rm i}\rho +1)e^{{\rm i}\theta},
\]
we find that $r=r_1$. By Remark \ref{circular-plane}, it suffices to show the existence of $s \in \mathbb{R}$ such that
\begin{center}
$\Delta x=$i$s \left( \frac{\partial x}{\| \partial x \|}+\frac{\partial x_1}{\| \partial x_1 \|} \right)$.
\end{center}
This is equivalent to showing
\begin{center}
$\arg \left( \Delta g -\overline{\Delta g}gg_1 \right) =\arg \left( \pm {\rm i} \left( \frac{g^{\prime}-\overline{g^{\prime}}g^2}{| g^{\prime}-\overline{g^{\prime}}g^2 |}+\frac{g_1^{\prime}-\overline{g_1^{\prime}}g_1^2}{| g_1^{\prime}-\overline{g_1^{\prime}}g_1^2 |} \right) \right)$,
\end{center}
which follows from
\[
\Delta g -\overline{\Delta g}gg_1=r(1+r^2)(e^{{\rm i}\theta_1}-e^{{\rm i}\theta}),
\]
\[
{\rm i} \left( \frac{g^{\prime}-\overline{g^{\prime}}g^2}{| g^{\prime}-\overline{g^{\prime}}g^2 |}+\frac{g_1^{\prime}-\overline{g_1^{\prime}}g_1^2}{| g_1^{\prime}-\overline{g_1^{\prime}}g_1^2 |} \right)
\]
\[
={\rm i} \left( \frac{{\rm i}\theta^{\prime}r(1+r^2)e^{\rm i \theta}}{| {\rm i} \theta^{\prime}r(1+r^2)e^{\rm i \theta} |}+
\frac{{\rm i}\theta_1^{\prime}r_1(1+r_1^2)e^{\rm i \theta _1}}{|{\rm i} \theta_1^{\prime}r_1(1+r_1^2)e^{\rm i \theta _1} |} \right)
\]
\[
=\pm (e^{{\rm i}\theta}-e^{{\rm i}\theta_1}),
\]
where we used the following lemma in the final equality above. This lemma follows from Lemma 6 and Theorem 11 in \cite{MW2}, because $g$ is isothermic.
\begin{lemma}
We have the following property:
\[
\theta^{\prime} \cdot \theta_1^{\prime}<0.
\]
\end{lemma}

We now prove the final sentence of Theorem \ref{semi-discrete weierstrass}. Let $x$ and be a semi-discrete minimal surface and $\psi$ be stereographic projection 
$\psi:\mathbb{S}^2 \rightarrow \mathbb{C}$. Then by definition, 
there exists a dual $x^*$ that is semi-discrete isothermic and inscribed in $\mathbb{S}^2$. Take
\[
g:=\psi \circ x^*,
\]
then $g$ is a semi-discrete holomorphic function (see Example 1 of \cite{MW2}). Setting
\[
x^*=(X_1,X_2,X_3), \ x^*_1=(X_{1,1},X_{2,1},X_{3,1}),
\]
we have
\[
g=\frac{X_1+{\rm i}X_2}{1-X_3}, \ X_1^2+X_2^2+X_3^2= X_{1,1}^2+X_{2,1}^2+X_{3,1}^2=1,
\]
\[
(X_1^{\prime})^2+(X_2^{\prime})^2+(X_3^{\prime})^2=\frac{\tau}{\nu ^2},
\]
\[
1-(X_1X_{1,1}+X_2X_{2,1}+X_3X_{3,1})=\frac{\sigma}{2\nu \nu_1}.
\]
Using the above equations and Definition \ref{duality}, computations give
\[
-\frac{\tau}{2} {\rm Re} \begin{pmatrix} \frac{1-g^2}{g^{\prime}}\\ \frac{{\rm i}(1+g^2)}{g^{\prime}} \\ \frac{2g}{g^{\prime}} \end{pmatrix}=
-\nu ^2\begin{pmatrix} X_1^{\prime} \\ X_2^{\prime} \\ X_3^{\prime} \end{pmatrix}=\partial x,
\]
\begin{equation}\label{x-described}
\frac{\sigma}{2} {\rm Re} \begin{pmatrix} \frac{1-g g_1}{\Delta g}\\ \frac{{\rm i}(1+g g_1)}{\Delta g}\\ \frac{g+g_1}{\Delta g} \end{pmatrix}=
\nu \nu_1 \begin{pmatrix} X_{1,1}-X_1 \\ X_{2,1}-X_2 \\ X_{3,1}-X_3 \end{pmatrix}=\Delta x.
\end{equation}

Thus $g$ produces $x$ via Equation (\ref{weierstrass-formula}), completing the proof.

Because the computation of (\ref{x-described}) in particular is rather laborious, we outline one part of that computation here: Since
\begin{eqnarray}
\frac{\sigma}{\Delta g}=\nu \nu_1 [ X_{1,1}(1-X_3)-X_1(1-X_{3,1}) \nonumber \\
-{\rm i}\{ X_{2,1}(1-X_3)-X_2(1-X_{3,1}) \} ], \nonumber
\end{eqnarray}
we have
\[
\frac{\sigma}{2}{\rm Re}\frac{1-gg_1}{\Delta g}=\frac{\nu \nu_1}{2(1-X_3)(1-X_{3,1})}{\rm Re}\left( \left[ X_{1,1}(1-X_3) \right. \right.
\]
\[
\left. -X_1(1-X_{3,1})-{\rm i}\{ X_{2,1}(1-X_3)-X_2(1-X_{3,1}) \} \right] \cdot
\]
\[
\left. [(1-X_3)(1-X_{3,1})-(X_1+{\rm i}X_2)(X_{1,1}+{\rm i}X_{2,1})] \right)
\]
\[
=\frac{\nu \nu_1}{2(1-X_3)(1-X_{3,1})} [\{ X_{1,1}(1-X_3)-X_1(1-X_{3,1})\} \cdot 
\]
\[
\{ (1-X_3)(1-X_{3,1}) -X_1 X_{1,1}+X_2 X_{2,1} \} -\{ X_{2,1}(1-X_3)
\]
\[
-X_2(1-X_{3,1}) \} \cdot (X_1 X_{2,1}+X_{1,1} X_2)]
\]
\[
=\frac{\nu \nu_1}{2(1-X_3)(1-X_{3,1})} \{ (1-X_3)(1-X_{3,1})(X_{1,1}-X_1
\]
\[
-X_{1,1} X_3+X_1 X_{3,1})-X_1(1-X_3)(X_{1,1}^2 +X_{2,1}^2)
\]
\[
+X_{1,1}(1-X_{3,1}) (X_1^2+X_2^2) \}
\]
\[
=\frac{\nu \nu_1}{2(1-X_3)(1-X_{3,1})} \{ (1-X_3)(1-X_{3,1})(X_{1,1}-X_1
\]
\[
-X_{1,1} X_3+X_1 X_{3,1})-(1-X_3)(1-X_{3,1})(1+X_{3,1})X_1
\]
\[
+(1-X_3)(1-X_{3,1})(1+X_3)X_{1,1} \}
\]
$=\nu \nu_1 (X_{1,1}-X_1)$.
\end{Proof}

We now give the proof of Lemma \ref{rotation}:

\begin{Proof}
From direct computation, one can check that
\begin{eqnarray}
-\frac{\tau}{2} {\rm Re} \begin{pmatrix} \frac{1-\hat{g}^2}{\hat{g}^{\prime}} \\ \frac{{\rm i}(1+\hat{g}^2)}{\hat{g}^{\prime}} \\ \frac{2\hat{g}}{\hat{g}^{\prime}}  \end{pmatrix}&=&
A\left( -\frac{\tau}{2}{\rm Re}\begin{pmatrix} \frac{1-g^2}{g^{\prime}} \\ \frac{{\rm i}(1+g^2)}{g^{\prime}} \\ \frac{2g}{g^{\prime}} \end{pmatrix} \right) ,  \nonumber\\
\frac{\sigma}{2} {\rm Re}\begin{pmatrix} \frac{1-\hat{g} \hat{g}_1}{\Delta \hat{g}} \\ \frac{{\rm i}(1+\hat{g} \hat{g}_1)}{\Delta \hat{g}} \\ \frac{\hat{g}+\hat{g}_1}{\Delta \hat{g}}  \end{pmatrix}&=&
A \left( \frac{\sigma}{2} {\rm Re}\begin{pmatrix} \frac{1-gg_1}{\Delta g} \\ \frac{{\rm i}(1+gg_1)}{\Delta g} \\ \frac{g+g_1}{\Delta g} \end{pmatrix} \right) , \nonumber
\end{eqnarray}
where 
\[
A=(a_{ij})_{i,j=1,2,3}
\]
with (we set $p=p_1 +{\rm i}p_2$, $q=q_1 +{\rm i}q_2$, $p_j \in \mathbb{R}$, $q_j \in \mathbb{R}$)
\[
a_{11}=p_1^2-p_2^2-q_1^2+q_2^2, \ a_{12}=-2p_1p_2-2q_1q_2,
\]
\[
a_{13}=-2p_1q_1 +2p_2q_2, \ a_{21}=2p_1p_2-2q_1q_2, 
\]
\[
a_{22}= p_1^2-p_2^2+q_1^2-q_2^2, \ a_{23}= -2p_1q_2-2p_2q_1,
\]
\[
a_{31}= 2p_1q_1+2p_2q_2, \ a_{32}= 2p_1q_2-2p_2q_1 
\]
\[
a_{33}=p_1^2+p_2^2-q_1^2-q_2^2.
\]
Lemma \ref{rotation} now follows.
\end{Proof}

\begin{example}
The semi-discrete minimal Enneper surface, has been given in \cite{MW2}. 
We can also obtain that surface by taking $g(k,t)=k+ {\rm i}t$ in Theorem \ref{semi-discrete weierstrass}.
\end{example}

\begin{example}
The $MW^{in}_{pd,rs}$ (resp. $MW^{in}_{ps,rd}$) catenoid can be constructed via Theorem \ref{semi-discrete weierstrass} with
\[
g(k,t)=c e^{\alpha k+{\rm i}\beta t} \ \ \ ({\rm resp.}\ \ g(k,t)=c e^{\alpha t+{\rm i} \beta k}),
\] 
for the right choices of $c, \ \alpha, \ \beta \in \mathbb{R} \setminus \{ 0\}$.
\end{example}

\section{Fully-discrete catenoids of Bobenko-Pinkall}

The fully discrete catenoids of Bobenko and Pinkall \cite{BP1} can be 
given explicitly by using the Weierstrass representation for discrete 
minimal surfaces (in the integrable systems sense), that is, we can use 
\begin{equation}
\label{BP-Wrep} x(q)-x(p) = \text{Re} \left( \frac{a_{pq}}{g_q-g_p} \begin{pmatrix}1-g_qg_p\\ \text{i}+\text{i}g_qg_p\\ g_q+g_p \end{pmatrix} \right)
\end{equation}
with the choice of $g$ as $g_p=g_{n,m}=c e^{c_1n+i c_2 m}$, where $c$, $c_1$, $c_2$ are nonzero real constants, and $p=(n,m)$ and $q=(n+1,m)$ or $q=(n,m+1)$, 
and $a_{pq}$ is a cross ratio factorizing function for $g$. This formulation can be found in \cite{BP1}, 
\cite{BP2}, \cite{H-J1}, \cite{Ho1}, \cite{Ro1}.

This choice of $g$ has cross ratios
\[
{\rm cr}(g_{n,m},g_{n+1,m},g_{n+1,m+1},g_{n,m+1})=\frac{-\sinh ^2 \frac{c_1}{2}}{\sin ^2 \frac{c_2}{2}}.
\]
So we can take $a_{pq}=-\alpha \sinh ^2 \frac{c_1}{2}$ (resp. $a_{pq}=\alpha \sin ^2\frac{c_2}{2}$), when $q=(n+1,m)$ (resp. $q=(n,m+1)$).  The value $\alpha \in \mathbb{R} \setminus \{ 0 \}$ can be chosen as we like.

Taking the axis of the surface to be
\[
\left\{ \begin{pmatrix} 0\\ 0\\ t \end{pmatrix} \Biggr| \ t \in \mathbb{R}  \right\} ,
\]
and taking the vertex in the profile curve at the neck to be $f(0,0)=(1,0,0)^t$, we can propagate to find the discrete profile curve in the $x_1x_3$-plane. For this 
purpose, $\alpha =-2$ and $c=-1$ are suitable values. One can check that, for all $m \in \mathbb{Z}$,
\[
x(0,m)=\begin{pmatrix}\cos (c_2 m)\\ \sin (c_2 m)\\0\end{pmatrix}.
\]

By (\ref{BP-Wrep}), the discrete profile curve in the $x_1x_3$-plane is, for all $n \in \mathbb{Z}$,
\[
x(n,0)=\begin{pmatrix} \cosh (c_1 n)\\ 0\\ n \cdot \sinh c_1 \end{pmatrix}.
\]
Again by (\ref{BP-Wrep}), we obtain
\[
x(n,m)=\begin{pmatrix} \cosh (c_1 n)\cos (c_2 m)\\ \cosh (c_1 n)\sin (c_2 m) \\ n \cdot \sinh c_1 \end{pmatrix}.
\]
Setting $\ell =\sinh c_1$, one profile curve of the $BP^{in}_{pd,rd}$-catenoid is as written in the upcoming Section \ref{proof of main theorem}. 
Note that the profile curves do not depend on $c_2$.

\begin{figure}
 \begin{center}

 \includegraphics[width=60mm]{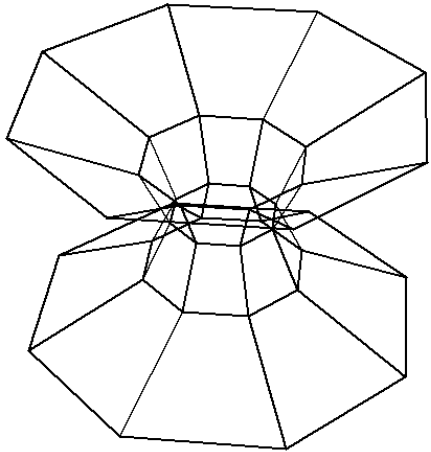}
  
 \end{center}
 \caption{a (fully-discrete) $BP^{in}_{pd,rd}$-catenoid}
 \label{BP-catenoid}
\end{figure} 

\section{Fully-discrete catenoids of Polthier-Rossman}

The catenoids described in \cite{PR1} are fully discrete and have discrete rotational 
symmetry, thus the symmetry group is a dihedral group.  Taking the
dihedral angle to be $\theta = 2 \pi K^{-1}$ for a constant $K \in \mathbb{N}$ and $K \geq 3$, 
the vertices of a profile curve (when the $x_3$-axis is the central 
axis of symmetry) in the $x_1 x_3$-plane can be taken to be points 
that are vertically equally spaced apart with height difference between 
adjacent vertices denoted as $\ell$, and the $x_1$ coordinates of the 
vertices can be taken as $x(n) = r \cdot \cosh (r^{-1} a \ell n)$, where 
$a = r \ell^{-1} {\rm arccosh} (1+r^{-2} \ell^2 (1+\cos \theta)^{-1})$.  
Here $r$ is the waist radius of the interpolated hyperbolic cosine curve. Taking $r=1$ without loss of generality, we can take one profile curve to be

\begin{figure}
 \begin{center}

 \includegraphics[width=60mm]{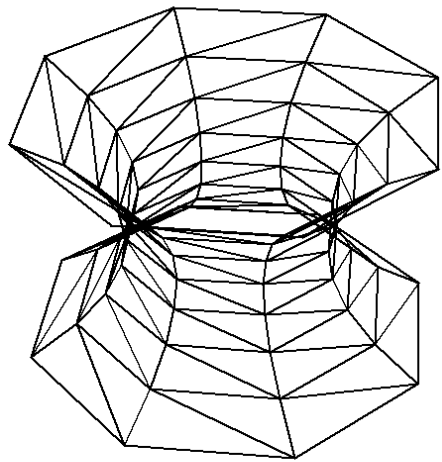}
  
 \end{center}
 \caption{a (fully-discrete) $PR^{va}_{pd,rd}$-catenoid}
 \label{PR-catenoid}
\end{figure}

\begin{equation}\label{PR-cat}
\begin{pmatrix} \cosh (n \cdot {\rm arccosh} (1+\ell^2 (1+\cos \theta)^{-1}))\\  
0\\ n \cdot \ell  \end{pmatrix} , \ n \in \mathbb{Z},
\end{equation}
so when we take the limit as $\theta \to 0$, we have 
\[
\begin{pmatrix}\cosh (n \cdot {\rm arccosh} (1+\tfrac{1}{2} \ell^2)) \\ 
0\\ n \cdot \ell \end{pmatrix}.
\] 
A direct computation, as in the proof of the next proposition, shows that this is exactly what was obtained 
by Machigashira \cite{Ma1}, although it was not described in terms of the 
hyperbolic cosine function there, but rather by using Chebyshev polynomials 
and Gauss hypergeometric functions.

\begin{proposition}\label{M-PR-limit}
The $M^{va}_{pd,rs}$-catenoid equals the limiting case of the $PR^{va}_{pd,rd}$-catenoids as 
$\theta \to 0$, and no $PR^{va}_{pd,rd}$-catenoid (with positive $\theta$) will 
ever have the same profile curve as the $M^{va}_{pd,rs}$-catenoid.  
\end{proposition}

\begin{Proof}
The vertices of an $M^{va}_{pd,rs}$-catenoid profile curve can be written as 
\begin{equation}\label{M-cat}
\begin{pmatrix} T_n(1+\frac{1}{2}\Lambda ^2)\\ 0\\ n \cdot \Lambda \end{pmatrix} \ ,
\end{equation}
where $T_k$ can be defined by the recursion
\[
T_0(z)=1, \ T_1(z)=z, \ T_n(z)=2zT_{n-1}(z)-T_{n-2}(z)
\]
for $n \geq 2$. The $T_n$ are called Chebyshev polynomials of the first kind, and are described in \cite{Ma1}. Suppose, for 
the vertex on the profile curve where $n=1$, we equate (\ref{PR-cat}) and (\ref{M-cat}), i.e.
\begin{equation}\label{PR-M-comparison}
\begin{pmatrix} 1+\tfrac{1}{2}\Lambda^2 \\ 0\\ \Lambda \end{pmatrix} = 
\begin{pmatrix} 1+\ell^2 (1+\cos \theta)^{-1}\\ 0\\ \ell \end{pmatrix} \; .
\end{equation}

The third coordinate in (\ref{PR-M-comparison}) implies $\Lambda = \ell$, and then the 
first coordinate implies $\theta=0$.  Then we would need to check that all other corresponding 
vertices in (\ref{PR-cat}) and (\ref{M-cat}) also become equal, to have proven the proposition.  

Letting $x$ denote the first coordinate of the profile curve, the 
$M^{va}_{pd,rs}$-catenoid satisfies
\[ x(n) = T_n (1+\tfrac{1}{2} \Lambda^2)=2 (1+\tfrac{1}{2} \Lambda^2) 
\cdot T_{n-1} (1+\tfrac{1}{2} \Lambda^2)
\]
\[ - 
T_{n-2}(1+\tfrac{1}{2} \Lambda^2). 
\]

For the limiting $PR^{va}_{pd,rd}$-catenoid ($\theta=0$), we would like to see the same recursion for the first 
coordinate of the profile curve, i.e. we wish to have
\[
\cosh (n \cdot {\rm arccosh}(1+\tfrac{1}{2} \ell^2)) = 
\]
\[
2 (1+\tfrac{1}{2} \ell^2) \cosh ( (n-1) \cdot {\rm arccosh} (1+\tfrac{1}{2} \ell^2)) 
\]
\[
- \cosh ((n-2) \cdot {\rm arccosh} (1+\tfrac{1}{2} \ell^2)) \; , 
\] and this is indeed true, proving the proposition. 
\end{Proof}

\section{Another type of semi-discrete catenoid}

Consider the two discrete loops, for a constant $K \in \mathbb{N}$, $K \geq 3$,
\[
\begin{pmatrix} \cos (2 \pi K^{-1})\\ \sin (2 \pi K^{-1})\\ \pm r\end{pmatrix}
\]
in the horizontal planes at height $\pm r$. We consider a 
semi-discrete catenoid (i.e. a surface with rotational symmetry by angle $2 \pi K^{-1}$ about the $x_3$-axis) 
with those two loops as boundary.  
This catenoid is comprised of $K$ congruent pieces, 
each piece foliated by horizontal line segments.  One such 
piece would have two boundary curves parametrized by 
\begin{equation}\label{neweqn20} 
c_1(t) = \begin{pmatrix} x(t)\\ 0\\ t\end{pmatrix} {\rm \ and \ }c_2(t) = \begin{pmatrix}x(t) \cos (2 \pi K^{-1})\\ x(t) \sin (2 \pi K^{-1})\\ t\end{pmatrix} 
\end{equation}
in vertical planes, with $t \in [-r,r]$ and with 
\[
x(r)=x(-r)=1.  
\]
The area of this piece is 
\[
A = \int_{-r}^r x \cdot \sqrt{2(1- \cos (2 \pi K^{-1}))+
(\sin (2 \pi K^{-1}))^2 (x^\prime )^2} dt . 
\]
Then consider a variation $x(t) \to x(t,\lambda)$ with $x(t,0)=x(t)$ 
and $x(\pm r,\lambda)=1$, so $\lambda$ is the variation parameter.  
Note that we are only considering rotationally invariant variations here, as was done by 
Machigashira \cite{Ma1}.  An interesting question that we do not address here is whether 
we are also in effect considering variations that are not rotationally invariant as well, by 
some semi-discrete version of the symmetric criticality principle, see \cite{Pa}.  Set
\[
x^{\prime}:=\frac{\partial x}{\partial t}, \ x_{\lambda}:=\frac{\partial x}{\partial t}, \ (x^{\prime})_{\lambda}:=\frac{\partial^2 x}{\partial \lambda \partial t}.
\] 

We wish to have that the following derivative with respect to $\lambda$ is zero, where $c:=\cos (2 \pi K^{-1})$ 
and $s:=\sin (2 \pi K^{-1})$ and $D:=2(1-c)+s^2 (x^\prime)^2$:
\[
\left. \frac{d}{d\lambda} A(\lambda ) \right|_{\lambda=0} = 
\left. \int_{-r}^r \frac{x_\lambda D +x x^\prime (x^\prime )_\lambda s^2}{\sqrt{D}} dt \right| _{\lambda=0}
\]
\[
= \int_{-r}^r \left( \hat{x} \sqrt{D}+s^2 \hat{x}^\prime 
\frac{1}{2} \frac{((x(t))^2)^\prime}{\sqrt{D}} \right) dt \; , 
\]
when $x(t,\lambda) = x(t) +\lambda \cdot \hat{x}(t) +{\mathcal{O}}(\lambda^2)$.  
Then, using integration by parts, we wish to have, with $x=x(t)$, 
\[  0 = \int_{-r}^r \hat{x} \left( \sqrt{D} - 
s^2 \frac{2(1-c) ((x^{\prime})^2+x x^{\prime \prime}) +s^2(x^{\prime})^4}{\sqrt{D}^3} \right) dt
\]
\[
=2 \int_{-r}^r \hat{x} \left( (1-c)^2 \frac{2-(1+c)(x x^{\prime \prime}-(x^{\prime})^2)}{\sqrt{D}^3} \right) dt
\]
for all variations, and this implies 
\[
x x^{\prime\prime} - (x^\prime)^2 = 2 (1+c)^{-1},
\]
with solutions ($c_2$, $c_3$ are free constants) 
\[
x=c_1 e^{-c_3 t} + c_2 e^{c_3 t} \ , \ \ \ c_1 = (2 (1+c) c_2 c_3^2)^{-1} \ .
\]
The conditions $x(\pm r)=1$ imply $c_1=c_2=(2 \cosh (c_3 r))^{-1}$, so we obtain
\begin{equation}\label{neweqn21}
x(t)=2c_1 \cosh (c_3 t)=\frac{\cosh (c_3 t)}{\cosh (c_3 r)}.
\end{equation}
Then automatically $x^\prime(0)=0$.

These catenoids have been determined here using a variational property, like the $M^{va}_{pd,rs}$-catenoids were, so we call them $M^{va}_{ps,rd}$-catenoids. 

From the above relations amongst the $c_j$, we obtain the following equation:
\begin{equation}\label{fromc1c2c3}
(\cosh (c_3 r))^2 =\frac{c_3 ^2 (1+c)}{2}.
\end{equation}
Thus $c_3$ is determined by $r$.  

For $r$ that allow for solutions $c_3$ to \eqref{fromc1c2c3}, 
a profile curve of an $M^{va}_{ps,rd}$-catenoid is $c_1(t)$ as 
in \eqref{neweqn20} with $x(t)$ as in \eqref{neweqn21}.  Rescaling 
this $c_1$ by $\cosh (c_3 r)$ and appropirately rescaling the parameter $t$, 
we find that this catenoid's profile curve can be parametrized as 
\[
\begin{pmatrix} \cosh \left( \frac{\sqrt{2}t}{\sqrt{1+\cos \frac{2 \pi}{k}}} \right) \\ 0 \\ t \end{pmatrix}.
\]

\begin{figure}
 \begin{center}

 \includegraphics[width=60mm]{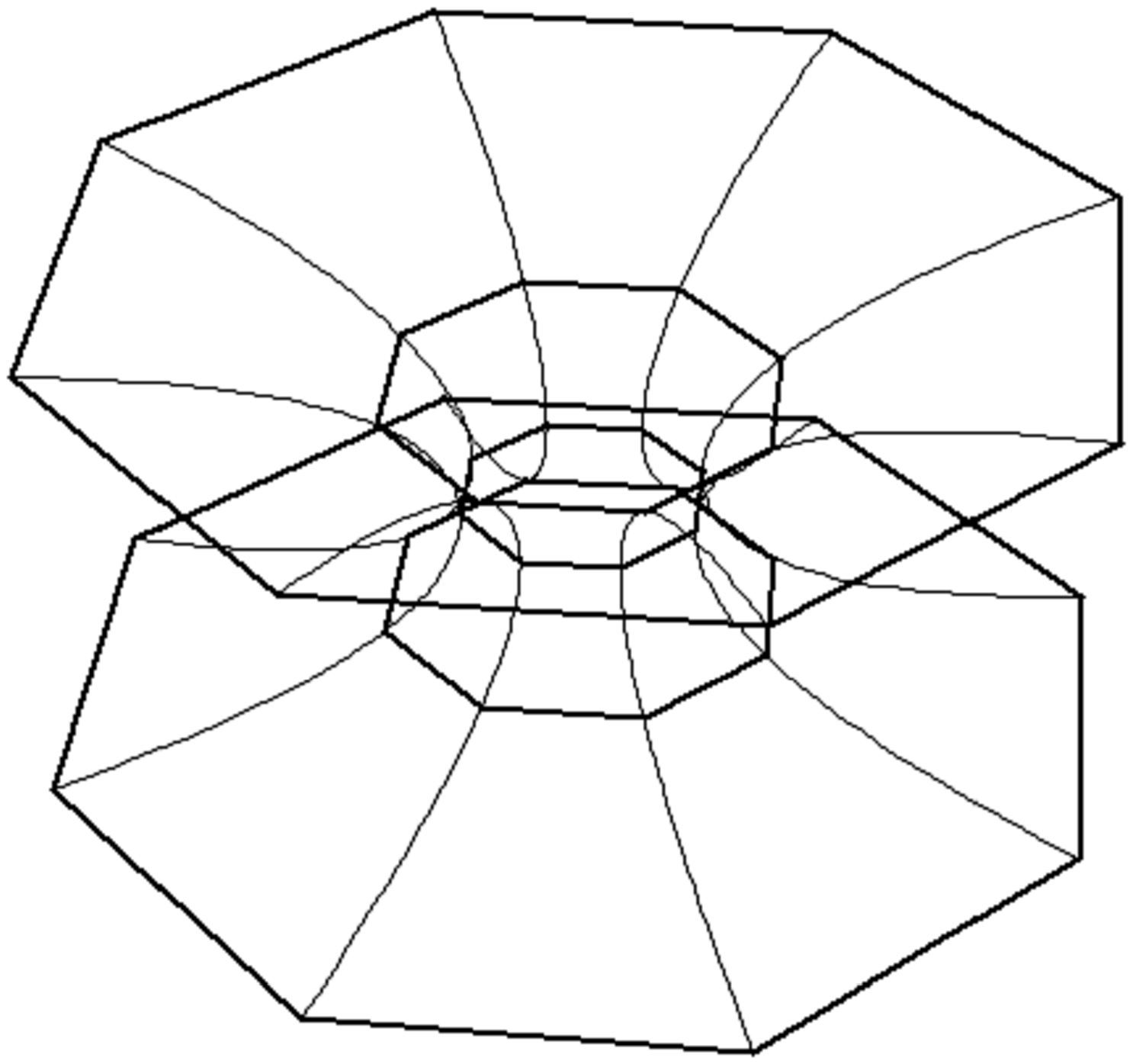}
  
 \end{center}
 \caption{a $M^{va}_{ps,rd}$-catenoid}
 \label{a $M^{va}_{ps,rd}$-catenoid}
\end{figure}

\section{proof of Theorem \ref{comparison}}\label{proof of main theorem}
We list parametrizations of the profile curves of the various catenoids again here in Table \ref{cat-parametrizations}.

Comparing all profile curves, we obtain the following proof of Theorem \ref{comparison}:

\begin{table}
\begin{center}
\begin{tabular}{|c|c|} \hline
 & parametrizations of profile curves \\ \hline
smooth catenoid & $\begin{pmatrix} \cosh t\\ 0\\ t\end{pmatrix} \ \ \ (t \in \mathbb{R})$ \\ \hline
$PR_{pd,rd}^{va}$-catenoid & $\begin{pmatrix} \cosh (n\cdot {\rm arccosh}(1+\frac{\ell^2}{1+\cos \frac{2 \pi}{K}} ))\\ 0\\ n\cdot \ell \end{pmatrix}$ \\ \hline
$M_{pd,rs}^{va}$-catenoid & $\begin{pmatrix} \cosh (n\cdot {\rm arccosh}(1+\frac{1}{2}\ell^2 ))\\ 0\\ n\cdot \ell \end{pmatrix}$ \\ \hline
$BP_{pd,rd}^{in}$-catenoid & $\begin{pmatrix} \cosh (n\cdot {\rm arcsinh}\ell )\\ 0\\ n\cdot \ell \end{pmatrix}$ \\ \hline
$MW_{pd,rs}^{in}$-catenoid & $\begin{pmatrix} \cosh (n\cdot {\rm arcsinh}\ell )\\ 0\\ n\cdot \ell \end{pmatrix}$ \\ \hline
$MW_{ps,rd}^{in}$-catenoid & $\begin{pmatrix} \cosh t\\ 0\\ t\end{pmatrix}$ \\ \hline
$M_{ps,rd}^{va}$-catenoid & $\begin{pmatrix} \cosh \left( \frac{\sqrt{2}t}{\sqrt{1+\cos \frac{2 \pi}{K}}} \right) \\ 0\\ t \end{pmatrix}$  \\ \hline
\end{tabular}
\caption{Parametrizations of seven types of catenoids}
\label{cat-parametrizations}
\end{center}
\end{table}

\begin{Proof}
The statements in items 1, 2, 4, 5 and 6 of Theorem \ref{comparison} are obvious, so we prove only item 3 here. By way of contradiction, suppose 
$BP^{in}_{pd,rd}$-catenoid profile curves and $M^{va}_{pd,rs}$-catenoid profile curves can be the same. 
From the parametrizations in Table \ref{cat-parametrizations}, 
\begin{equation}\label{BP-M}
\cosh ({\rm arc}\sinh \ell )=1+\frac{1}{2}\ell ^2.
\end{equation}
Since $\cosh ({\rm arc}\sinh \ell )=\sqrt{1+\ell ^2}$, \eqref{BP-M} implies $\ell =0$, which does not occur. Similarly, 
suppose $BP_{pd,rd}^{in}$-catenoid profile curves and $PR_{pd,rd}^{va}$-catenoid profile curves can be the same.  Then 
\[
\cosh ({\rm arc}\sinh \ell )=1+\frac{1}{1+\cos \frac{2 \pi}{K}}\ell ^2
\]
\begin{equation}\label{BP-PR}
\Leftrightarrow \left( \frac{-2}{1+\cos \frac{2 \pi}{K}}+1 \right) \ell^2 =\left( \frac{1}{1+\cos \frac{2 \pi}{K}} \right) ^2 \ell ^4.
\end{equation}
The left-hand-side of (\ref{BP-PR}) is negative and the right-hand-side of (\ref{BP-PR}) is positive, which is impossible, proving the theorem.
\end{Proof}

\end{document}